\documentclass{article}

\usepackage{amssymb,amsmath,amsfonts,amsthm}
\usepackage[letterpaper, margin=1in]{geometry}
\usepackage{graphicx}
\usepackage{url}
\usepackage{subfigure}
\usepackage{float}
\usepackage[latin1]{inputenc}
\usepackage[usenames]{color}

\usepackage{hyperref}
\newtheorem{theorem}{Theorem}[section]
\newtheorem{corollary}[theorem]{Corollary}
\newtheorem{lemma}[theorem]{Lemma}
\newtheorem{example}[theorem]{Example}
\newtheorem{remark}[theorem]{Remark}

\usepackage{graphicx,mathrsfs,amsmath}
\newcommand{\FF}{\mathcal{F}}

\newcommand{\oGam}{\circ(\Gamma)}

\newcommand{\xGam}{\otimes(\Gamma)}

\newcommand{\bGam}{\bullet(\Gamma)}
\newcommand{\bGami}{\bullet(\Gamma_i)}

\newcommand{\setof}[2]{\left\{ #1 \, : \, #2 \right\}}
\newcommand{\N}{{\mathbb N}}

\newcommand{\Z}{{\mathbb Z}}

\newcommand{\tabulatedset}[1]{\left\{ #1 \right\}}
\newcommand{\C}[2]{C_{#1}^{#2}}

\newtheorem{definition}[theorem]{Definition}
\begin{document}

\begin{center}
         \bfseries\Large Circuits and circulant minors\footnote{Partially supported by PIP-CONICET 277, PID-UNR 416, PICT-ANPCyT 0586, and MathAmSud 15MATH06 PACK-COVER.}
 \end{center}

\smallskip

\begin{center}
\large Silvia Bianchi$^{a}$, Graciela Nasini$^{a,b}$, Paola Tolomei$^{a,b}$, and Luis Miguel Torres$^{c}$\\
\normalsize \{\texttt{sbianchi,nasini,ptolomei}\}\texttt{@fceia.unr.edu.ar, luis.torres@epn.edu.ec} \vspace{0.5cm}\\
\small{$^{a}$FCEIA, Universidad Nacional de Rosario, Rosario, Argentina}\\
\small{$^{b}$CONICET - Argentina}\\
\small{$^{c}$Centro de Modelizaci{\'o}n Matem{\'a}tica - ModeMat, Escuela Polit{\'e}cnica Nacional, Quito, Ecuador}\\ 
\end{center}

\medskip

\begin{abstract} 
Circulant contraction minors play a key role for characterizing ideal circular matrices in terms of
minimally non ideal structures. In this article we prove necessary and sufficient conditions for
a circular matrix $A$ to have circulant contraction minors in terms of circuits in a digraph associated with $A$.  In the particular case when $A$ itself is a
circulant matrix, our result provides an alternative characterization to the one previously known from
the literature. 

 \noindent\textit{keywords:} circular matrices, circulant minors, circuits, idealness. 

\end{abstract}

\section{Introduction}
Given a set $E=\{1,2,\ldots, n\}$ and a family $\FF$ of subsets of $E$, a \emph{packing} or \emph{covering} of $\FF$ is defined as a set $S \subset E$ that intersects each member of $\FF$ \emph{at most} or \emph{at least} in one element, respectively. If a weight is associated with each element of $E$, then the \emph{set packing problem} (SPP) asks for finding a packing of maximum weight, while the \emph{set covering problem} (SCP) asks for a minimum-weight covering. A wide range of  problems in combinatorics and graph theory can be formulated as set packing or set covering problems.

Both problems are known to be NP-hard in general. A common approach for their study consists in formulating them as integer linear programs. Let $M(\mathcal F)$ be a $0,1$-matrix whose rows are the incidence vectors of the members of $\FF$, and let $w\in \Z_+^n$. Then the problems can be formulated as:
\begin{align*}
\textrm{(SPP)} & \qquad \max\{w^Tx: M(\mathcal F) x\leq {\bf 1}, x\in \{0, 1\}^n\} &&&&&&&&
\textrm{(SCP)} &  \qquad \min\{w^Tx: M(\mathcal F) x\geq {\bf 1}, x\in \{0, 1\}^n\},
\end{align*}
\noindent where ${\bf 1} \in \Z^m$ is the vector whose entries are all equal to one.
   
Despite of their seeming similarity, it has been pointed out that these two problems have strong structural
differences. The set packing problem has been shown to be equivalent to the maximum-weight stable set problem, and this equivalence can be exploited for obtaining characterizations, strengthening formulations, and devising solution algorithms. In contrast, the set covering problem does not seem to have an equivalent representation as a graph optimization problem, even if coverings play an important role in the formulation of several important graph problems such as connectivity, coloring, and dominating sets, to cite some examples. As a consequence, the set covering problem has been far less studied than the set packing problem.

One important question is to characterize such families $\FF$ for which the integer programing formulations SPP and SCP are \emph{perfect formulations}, i.e., the linear systems  $M(\mathcal F) x \leq {\bf 1}, 0 \leq x \leq {\bf 1}$ and  $M(\mathcal F) x \geq {\bf 1}, 0 \leq x \leq {\bf 1}$ provide complete linear descriptions of the convex hulls of all feasible solutions of the corresponding problems. This question was solved in \cite{Chvatal75} for the case of set packing, using results from the (weak) perfect graph theorem: the formulation SPP is a perfect formulation if and only if $\FF$ is the family of maximal cliques of a perfect graph. Accordingly, maximal clique-vertex incidence matrices of perfect graphs are termed as \emph{perfect matrices}. On the other hand, $0,1$-matrices $M(\mathcal F)$ for which SCP is a perfect formulation for the set covering problem are known as \emph{ideal} matrices, but they have not yet been completely characterized.

Perfectness is a hereditary graph property, which means that any vertex induced subgraph of a perfect graph is itself perfect, and the corresponding holds for their clique-vertex incidence matrices. Similarly,  idealness can be shown to be a hereditary matrix property, which is transferred to \emph{minors} of the matrix. In the first case, this observation has led to the characterization of perfect graphs in terms of minimally non perfect subgraphs. The corresponding characterization of minimally non ideal matrices turned out to be much more difficult and is still an open task. However, several results have been obtained for particular classes of matrices.

Cornu\'ejols and Novick \cite{Cornuejols94} have characterized all ideal and minimally non ideal \emph{circulant} matrices. Circulant minors of a given circulant matrix play a fundamental role in this characterization. This fact motivated the authors to study conditions for such a minor to exist. They provided a sufficient condition in terms of the existence of a simple directed circuit in a particular digraph 
associated with the matrix. Later, Aguilera \cite{Aguilera09} extended this result, obtaining necessary and sufficient conditions in terms of the existence of a family of disjoint directed circuits in the same auxiliary digraph 
.

\emph{Circular} matrices generalize circulant matrices. Eisenbrand \emph{et al.} \cite{EisenbrandEtAl08} obtained a perfect formulation for SPP when $M(\FF)$ is a circular matrix. 
The inequalities involved in this formulation are related to directed circuits in another auxiliary digraph 
associated with the matrix. More recently, we have obtained a perfect formulation for SCP  
\cite{nosotres}. Once again, directed circuits in a certain digraph are related to the inequalities that appear in the linear description. Furthermore, all relevant directed circuits in our case induces circulant minors. As a consequence, non-ideal circulant minors are the minimal structures necessary to avoid idealness of circular matrices.

In \cite{nosotres} we also stated a necessary condition for a circular matrix to have a circulant minor. In this paper we further develop this result and completely characterize circulant minors of circular matrices in terms of directed circuits in its associated digraph. When restricted to the subclass of circulant matrices, our result yields an alternative characterization of circulant minors to the one provided in \cite{Aguilera09}.

\section{Notations and preliminary results} \label{intro}

For $n \in \N$, $[n]$ will denote the additive group defined on the set
$\tabulatedset{1, \ldots, n}$, with integer addition modulo $n$. 
Given $a,b\in[n]$, let $b-a$ be the minimum non-negative integer $t$ such that $a+t=b \mod n$. We denote by $[a,b]_n$ the \emph{circular interval} defined by the set $\{a+s: 0\leq s \leq b-a\}$. Similarly, $(a,b]_n$, $[a,b)_n$, and $(a,b)_n$ correspond to $[a,b]_n\setminus \{a\}$, $[a,b]_n\setminus \{b\}$, and $[a,b]_n\setminus \{a,b\}$, respectively. 


Unless otherwise stated, throughout this paper $A$ denotes a $\{0,1\}$-matrix of order $m\times n$.
Moreover, we consider the columns (resp. rows) of $A$ to be indexed by
$[n]$ (resp.~by $[m]$).
Two matrices $A$ and $A'$ are
\emph{isomorphic}, written as $A\approx A'$, if $A'$ can be
obtained from $A$ by a permutation of rows and columns. 

In the context of this paper, a matrix $A$ is called \emph{circular} if, for every row $i \in [m]$,
there are two distinct integer numbers $\ell_i, u_i\in [n]$ such that the $i$-th row of $A$ is the incidence vector of the set $[\ell_i,u_i]_n$.

A row $i$ of a circular matrix $A$ is said to \emph{dominate} a row $j \neq i$ of $A$ if the set $[\ell_j,u_j]_n \subseteq [\ell_i,u_i]_n$. Moreover, a row is dominating if it dominates some other row. In the following, we restrict our attention to matrices without dominating rows and without zero rows or columns. Interval matrices are a particular case of circulant matrices and it is known that they are ideal.

The following is an example of a $6\times 12$-circular matrix.
\begin{equation}
A=\left(
\begin{array}{cccccccccccc}
1&1&1&1&1&0&0&0&0&0&0&0\\
0&1&1&1&1&1&1&1&0&0&0&0\\
0&0&0&0&1&1&1&1&1&0&0&0\\
0&0&0&0&0&0&1&1&1&1&0&0\\
0&0&0&0&0&0&0&0&0&1&1&1\\
1&1&0&0&0&0&0&0&0&0&0&1\\
\end{array}
\right)
\label{ejemplo}
\end{equation}

A square circular matrix of order $n$ is called
a \emph{circulant matrix}. Observe that in this case the sets $[\ell_i,u_i]_n$ must have the same cardinality, say $k$, for
all $i\in [n]$, with $k \geq 2$.
Such a matrix will be denoted by $\C{n}{k}$ and 
w.l.o.g. we can assume that, for every $i\in [n]$, the $i$-th row of $\C{n}{k}$ is the incidence vector of the set $[i,i+k)_n$. 

Given $N\subset [n]$, the \emph{minor of} $A$ \emph{obtained by contraction of} $N$, denoted by $A/N$, is the submatrix of $A$ that results after removing all columns with indices in $N$ and all dominating rows. Moreover, the \emph{minor of} $A$ \emph{obtained by deletion of} $N$, is the submatrix of $A$ that results after removing all columns with indices in $N$ and all rows having an entry equal to $1$ in some column indexed by $N$. It is not hard to see that every proper minor of a circular matrix obtained by deletion is an interval matrix and then, it is ideal. As we are interested in non-ideal minors of circular matrices, in this work we focus only on minors obtained by contraction, and refer to them simply as \emph{minors}. Moreover, a minor of a matrix $A$ is called a \emph{circulant minor} if it is isomorphic to a circulant matrix.  

Circulant minors of circulant matrices have an interesting combinatorial characterization in terms of circuits in a particular digraph. Indeed, given a circulant matrix $\C{n}{k}$, the authors in \cite{Cornuejols94} define a directed auxiliary graph $G(\C{n}{k})$ with $[n]$ as its set of vertices and arcs of the form $(i,i+k)$ and $(i,i+k+1)$ for every $i\in[n]$, i.e., all arcs \emph{of length} $k$ and $k+1$, respectively. They prove that if $N\subset [n]$ induces a simple circuit in $G(\C{n}{k})$, then the matrix $\C{n}{k}/N$ is isomorphic to a circulant minor. In a subsequent work, Aguilera \cite{Aguilera09} shows that $C_n^k/N$ is isomorphic to a circulant minor of $\C{n}{k}$ if and only if $N$ induces  
a family of disjoint simple circuits in $G(\C{n}{k})$, each one having the same number of arcs of length $k$ and the same number of arcs of length $k+1$.

Working on the set covering problem on circular matrices \cite{nosotres}, we have found a sufficient condition for a circular matrix to have a circulant minor, also expressed in terms of circuits in the following digraph associated with the matrix: 

\begin{definition}\label{D(A)}\cite{nosotres}
Given a circular matrix $A$, let $F(A)$ be the directed graph whose set of vertices is $[n]$ and whose arcs are  
of the form $(\ell_i -1, u_i)$, for every $i\in [m]$ (called \emph{ row arcs}) and $(j, j-1), (j-1, j)$ with $j \in [n]$ (termed as \emph{reverse short arcs} and \emph{forward short arcs},  respectively).
\end{definition}   

We say  that a row arc $(u,v)$ in $F(A)$ jumps over a vertex $j \in [n]$ if $j \in (u  ,v]_n$. Moreover, the only forward (resp. reverse) short arc jumping over $j$ is the arc $(j-1,j)$ (resp. $(j,j-1)$).

Given a row arc $a=(u,v)$ of $F(A)$ the length of $a$, denoted by $l(a)$, equals $v-u$. If $a$ is a short arc, then $l(a)=1$ if it is a forward arc and $l(a)=-1$ if it is a reverse arc. The \emph{winding number} $p(\Gamma)$ of a directed circuit $\Gamma$ in $F(A)$ is defined by
$$p(\Gamma)= \frac{\sum\limits_{a\in E(\Gamma)}l(a)}{n},$$ 
where $E(\Gamma)$ denotes the set of arcs of $\Gamma$.

Any circuit $\Gamma$ in $F(A)$ induces a partition of the vertices of $F(A)$ into the following three classes:
\begin{itemize}
\item[(i)] \emph{circles} 
$\oGam := \setof{j \in [n]}{(j-1,j)\in E(\Gamma)}$,
\item[(ii)] \emph{crosses} 
$\xGam :=\setof{j \in [n]}{(j, j-1)\in E(\Gamma)}$, and 
\item[(iii)] \emph{bullets} $\bGam: = [n]\setminus (\circ(\Gamma)\cup \otimes(\Gamma))$.  
\end{itemize}

Figure~\ref{FdeAyGama} shows the digraph $F(A)$ for matrix $A$ in 
(\ref{ejemplo}). For illustration, a circuit 
$$\Gamma:=\{2,3,4,9,12,1,8,7,6,10,11,2\}$$ is depicted in bold lines. It has five row arcs and winding number two. 
Moreover, it induces the following partition of the vertices of  $F(A)$:  $\oGam=\{1,3,4,11\}$,  $\xGam=\{7,8\}$ and $\bGam= \{2,5,6,9,10,12\}$.

\begin{figure}[htbp]
	\centering
		\includegraphics[width=0.3\textwidth]{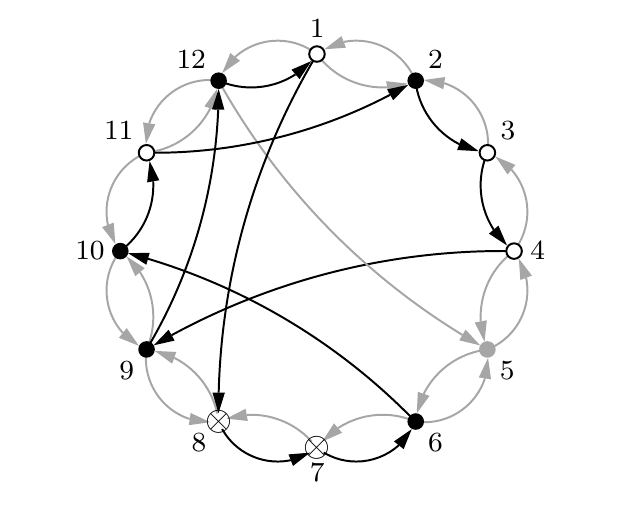}
	\caption{Auxiliary digraph $F(A)$ and a circuit $\Gamma$.}
	\label{FdeAyGama}
\end{figure}

Observe that circle (resp. cross) vertices are the heads (resp. tails) of forward (resp. reverse) short arcs of $\Gamma$. A bullet vertex is either a vertex outside $\Gamma$, or it is the tail or the head of a row arc.
We say that a bullet is an \emph{essential bullet} if it is reached by $\Gamma$. In Figure \ref{FdeAyGama} all vertices in $\bGam$ except for vertex $5$ are essential bullets.
 
In the following we denote by $p$ the winding number of $\Gamma$ and assume that the circuit has $s$ essential bullets $\{b_j: j\in [s]\}$, with $1\leq b_1 < b_2<\cdots< b_s\leq n$. 

For $j\in [s]$, let $t_j:=\min\{t\geq 1: b_j+t\in \bGam\}$. It can be verified that, if $t_j\geq 2$, then $[b_j+1,b_{j}+t_j)_n\subset \oGam$ or $[b_j+1,b_{j}+t_j)_n\subset \xGam$ (see \cite{nosotres} for further details). Then, denoting by $v_j$ the vertex $b_j+t_j-1$, we define the \emph{block} $B_j:=[b_j, v_j]_n$ which can be a \emph{circle} block, a \emph{cross} block or a \emph{bullet} block, depending on the vertex class that $b_j+1$ belongs to. 

It is straightforward to see that the blocks $\{B_j: j\in [s]\}$ define a partition of the vertex set of $\Gamma$. Moreover, for each $j\in [s]$, there exists one row arc \emph{leaving} $B_j$ and another row arc \emph{entering} $B_j$. 
Let $B_j^{-} \in B_j$ be the tail of the arc leaving $B_j$, while $B_j^{+} \in B_j$ denotes the head of the arc entering  $B_j$. In particular, if $B_j$ is a cross block, $B^-_j=b_j$ and $B^+_j=v_j$; if $B_j$ is a circle block, $B^-_j=v_j$ and $B^+_j=b_j$; finally, if $B_j$ is a bullet block, $B^-_j=B^+_j=b_j=v_j$. 

In the circuit from the example in Figure \ref{FdeAyGama}, the essential bullets are $b_1=2, b_2=6, b_3=9, b_4=10$, and $b_5=12$. The block
 $B_2=[6,8]$ is a cross block, $B_3=[9,9]$ is a bullet block and $B_1=[2,4]$, $B_4=[10,11]$ and $B_5=[12,1]$ are circle blocks. Observe that $B^+_1=b_1=2$ while $B^+_2=v_2=8$.

We gather some of the results in \cite{nosotres} in the following theorem:  

\begin{theorem}\cite{nosotres}
\label{peb}
Let $A$ be a circular matrix and $\Gamma$ be a circuit of $F(A)$ with winding number $p$ and $s$ essential bullets $\{b_j: j\in [s]\}$, with $1\leq b_1 < b_2<\cdots< b_s\leq n$. Then, $\gcd(s,p)=1$ and the row arcs of $\Gamma$ are $(B^-_i,B^+_{i+p})$ with $i\in [s]$, i.e. each row arc of $\Gamma$ jumps over $p$ essential bullets.

In addition, if $(u,v)$ is a row arc of $F(A)$ that jumps over $k$ essential bullets of $\Gamma$, then $k \in \{p-1, p, p+1\}$. Moreover, if $k=p-1$ (resp. $k=p+1$) then $u$ (resp. $v$) is a vertex of $\Gamma$.   
\end{theorem}

%
%
%

We say that a row arc in $F(A)$ is a \emph{bad arc} (\emph{with respect to} $\Gamma$) if it jumps over $p-1$ essential bullets of $\Gamma$.
In Figure~\ref{FdeAyGama}, the row arc $(12,5)$ is a bad arc with respect to $\Gamma$ since it jumps only over one essential bullet, namely vertex $b_1=2$, while the winding number of $\Gamma$ is two.
In \cite{nosotres} it is proved that if $(u,v)$ is a bad arc of $\Gamma$ then $u$ belongs to a circle block and $v$ is either a circle or it is not reached by $\Gamma$.

The following theorem gives a sufficient condition for a circular matrix to
have a circulant minor:

\begin{theorem} \label{uncircuit}\cite{nosotres}
Let $A$ be a circular matrix. A circuit $\Gamma$ in $F(A)$ with $s$ row arcs, winding number $p$, and 
without bad arcs induces a circulant minor $C_s^p$. More precisely, if $B$ 
is the set of essential bullets of $\Gamma$ and $N=[n]\setminus B$ then $A/N \approx C_s^p$.
\end{theorem}

As an illustration of the previous theorem, we present the following example.
\begin{example}\label{menorycircuito}
Consider the circular matrix $A$ given in (\ref{ejemplo}). The sequence of vertices 
$$(2,3,4,9,12,5,6,10,11,2)$$ induces a circuit 
in $F(A)$ without bad arcs,  
having five row arcs, and winding number two. The set of its essential bullets is $B=\{2,5,9,10,12\}$ and $N=[12]\setminus B$. It is easy to check that $A/N \approx C_5^2$.
\end{example}

Clearly, from the theorem above and Theorem~\ref{peb}, simple directed circuits without bad arcs in $F(A)$ \emph{induce} circulant minors $C_s^p$ of $A$, with $\gcd(s,p)=1$. We will see that, in order to obtain circulant minors $C_s^p$ with $\gcd(s,p)\geq 2$, \emph{families} of circuits in $F(A)$ are needed. Let us introduce for this purpose some more notations and definitions.


Let $\Gamma=\{\Gamma^i: i\in[a]\}$ be a family of $a$ (vertex) disjoint circuits in $F(A)$ and let 
$\bGam:=\bigcup_{i\in[a]} \bGami$. The set of essential bullets of $\Gamma$, $\oGam$ and $\xGam$ are defined in the same way. 
We say that a row arc in $F(A)$ is a \emph{bad arc} \emph{with respect to} $\Gamma$, if it is a bad arc with respect to $\Gamma^i$, for some $i\in[a]$.  

For $i\in [a]$, let $p_i$ and $s_i$ be the winding number and the number of row arcs of $\Gamma^i$, respectively. In addition, let $\{b^i_{\ell}: \ell\in[s_i]\}$ and  $\{B^i_{\ell}=[b_{\ell}^i,v_{\ell}^i]: \ell\in [s_i]\}$ denote the set of essential bullets and the partition of vertices of $\Gamma^i$ into blocks, respectively. 

In the next section we show that every family of disjoint circuits in $F(A)$ with no bad arcs induces a circulant minor of $A$.

\section{From circuits to circulant minors}

In the following the next known result on digraphs will be useful.  
\begin{remark}\label{ciclos}
Let $D$ be a digraph with vertex set $[s]$ and arcs of the form $(i,i+p)$ for all $i\in [s]$, with $1\leq p\leq n-1$.
Let $a=\gcd(s,p)$. Then $D$ is a collection of $a$ disjoint circuits, each one with $\frac{s}{a}$ arcs and winding number $\frac{p}{a}$.
\end{remark}

The next theorem proves that all circuits in a family of disjoint circuits of $F(A)$ have the same number of row arcs and the same winding number.

\begin{theorem}\label{mismopys}
Let $\Gamma=\{\Gamma^i: i\in[a]\}$ be a family of $a$ disjoint circuits in $F(A)$, each one with $s_i$ row arcs and winding number $p_i$. If $s=\sum\limits_{i=1}^a s_i$ and $p=\sum\limits_{i=1}^a p_i$ then $s_i=\frac{s}{a}$ and $p_i=\frac{p}{a}$ holds for all $i\in [a]$.  Moreover, $a = \gcd(s,p)$, each row arc of $\Gamma$ jumps over $p$ essential bullets, and no pair of row arcs in $\Gamma$ jumps over the same set of essential bullets. 
\end{theorem}

\begin{proof}
Observe that $\Gamma$ has $s=\sum\limits_{i=1}^a s_i$ disjoint row arcs. 

Given $i\in [a]$, by Theorem \ref{peb}, a row arc of $\Gamma^i$  jumps over $p_i$ essential bullets of $\Gamma^i$ and . $p_j$ bullets of $\Gamma^j$, for all $j\in [s]$, $j\neq i$, since $\Gamma^i$ and $\Gamma^j$ are disjoint.  

Then any arc of $\Gamma$ jumps over $p=\sum\limits_{h=1}^a p_h$ essential bullets of $\Gamma$. 

Let $\{b^i_{\ell}: \ell\in[s_i]\}$ be the set of essential bullets of $\Gamma^i$. From Theorem \ref{peb} we have that every arc of $\Gamma$ joins a vertex in the block $B^i_{\ell}=[b_{\ell}^i,v_{\ell}^i]_n$ with a vertex in the block $B^i_{\ell+p_i}=[b_{\ell+p_i}^i,v_{\ell+p_i}^i]_n$, for some $\ell\in [s_i]$ and  $i\in [a]$.

Now assume the $s$ essential bullets of $\Gamma$ are relabeled in such a way that $1\leq b_1<b_2<\cdots<b_s\leq n$ and let $B_j$ be the block containing  $b_j$. Since every row arc of $\Gamma$ jumps over $p$ essential bullets, we have that every row arc of $\Gamma$ goes from block $B_j$ to block $B_{j+p}$, for some $j\in [s]$ and thus, it jumps over the $p$ essential bullets $\{b_{j+1},\ldots,b_{j+p}\}$. Hence, no pair of row arcs jumps over the same set of essential bullets. 

Let $D$ be the directed digraph obtained by \emph{shrinking} every block in $\Gamma$ into its corresponding essential bullet, i.e, $D$ has $\{b_j:j\in [s]\}$ as vertex set and arcs of the form $(b_i,b_{i+p})$ for every $i\in [s]$.  
From Remark \ref{ciclos}
$D$ consists of $\gcd(s,p)$ disjoint circuits, each one having $\frac{s}{a}$ row arcs and winding number $\frac{p}{a}$.  
%
From the one-to-one correspondence between the row arcs of $\Gamma$ and the arcs of $D$, it follows that $\Gamma$ is a collection of $\gcd(s,p)$ disjoint circuits, each one having $\frac{s}{a}$ row arcs and  winding number equal to $\frac{s}{a}$.
\end{proof}

As a consequence, we obtain the following sufficient condition for a circular matrix to have a circulant minor. 

\begin{theorem}
\label{th:circulant-minor}
Let $A$ be a circular matrix and $\Gamma=\{\Gamma^i: i\in[a]\}$ be a family of $a$ disjoint circuits in $F(A)$ without bad arcs, each one having $\tilde s$ row arcs and winding number $\tilde p$
. In addition, let $B\subset [n]$ be the set of essential bullets of $\Gamma$, $N:=[n]\setminus B$, 
$s=a \tilde s$, and $p= a\tilde p$. Then $A/N \approx C_{s}^{p}$.
\end{theorem}

\begin{proof}
%
As in the proof of Theorem \ref{mismopys}, assume the vertices in $B$ are relabeled in a such a way that $1\leq b_1<b_2<\cdots<b_s\leq n$. Let $A'$ be the submatrix of $A$ whose rows are in correspondence with the row arcs in $\Gamma$ and whose columns are indexed by the vertices in $B$. It follows that $A'$ is a $s\times s$-matrix. Moreover, since every row arc of $\Gamma$ jumps over $p$ consecutive vertices in $B$ and no pair of row arcs jumps over the same set of vertices, 
each row of $A'$ is the incidence vector of a circular interval of the form $[i, i+p)_s$, with $i \in [s]$, and no two rows of $A'$ are identical to each other.
Then, $A'$ is isomorphic to $C_s^p$. Finally, since $\Gamma$ has no bad arcs, each row of $A$ not in correspondence with an arc of $\Gamma$ has at least $p$ entries equal to one, i.e., it dominates some row from $A'$. Then, $A' = A/N$.
\end{proof}

In the digraph $F(A)$ depicted in Figure~\ref{FdeAyGama}, consider the family $\Gamma:= \{\Gamma^1,\Gamma^2\}$ containing the two circuits induced by the sequences of vertices $(1,8,7,6,10,11,2,1)$ and $(4,9,12,5,4)$, respectively. Each circuit has three row arcs and winding number equal to one. Moreover, $\Gamma$ has no bad arcs. It is easy check that the set of essential bullets of $\Gamma$ is $B=\{1,4,6,9,10,12\}$ and that $A/N \approx C_6^2$.


\section{From circulant minors to circuits}

It is natural to ask whether the converse of Theorem \ref{th:circulant-minor} holds. 
In other words, whether, given $N\subset [n]$ such that $A/N \approx C_s^p$,  there is a collection $\Gamma$ of disjoint circuits in $F(A)$ such that the set $B=[n]-N$ corresponds to the essential bullets of $\Gamma$.  The following examples show that this is not the case in general.

Consider the matrix $A$ given in \eqref{ejemplo} and let $B=\{2,5,7,10,12\}$. We have $A/N \approx C_5^2$, where $N=[12]\setminus B$. However, $F(A)$ contains no circuit for which vertex $7$ is an essential bullet, since no row arc in $F(A)$ has vertex $7$ as its head or tail.

As a second example, consider the set $B=\{2,5,8,10,12\}$. Again, we have $A/N \approx C_5^2$, with $N=[12]\setminus B$. In this case, each vertex in $B$ is reached by at least one row arc in $F(A)$. Now assume there is a family $\Gamma$ of circuits for which $B$ is the set of essential bullets. Then the row arc $(1,8)$ must belong to $E(\Gamma)$,
as it is the only row arc in $F(A)$ reaching vertex $8$. However, $(1,8)$ jumps over the three essential bullets $\{2,5,8\}$, contradicting Theorem \ref{th:circulant-minor}.


In the following, let $A$ be a circular matrix and $N\subset [n]$ such that $A/N \approx C_s^p$. Moreover, let $B=[n]\setminus N=\{b_1,\ldots,b_s\}$, with $1\leq b_1<\ldots<b_s\leq n$. 

As $A/N \approx C_s^p$ is a minor of $A$, for every $j\in B$ there is at least one $i\in[m]$ for which $[\ell_i,u_i]_n\cap B=\{b_{j-p+1},\ldots, b_j\}$. Observe that this row  is not necessarily unique. Indeed, for $A$ as given in \eqref{ejemplo}
and $B=\{2,5,8,10,12\}$, both rows $1$ and $2$ intersect $B$ in the same set $\{2,5\}$.

For each $j\in B$, let $R(j):=\{i\in [m]: [\ell_i,u_i]_n\cap B=\{b_{j-p+1},\ldots, b_j\}\}$. Clearly, any submatrix of $A$ obtained by selecting one row in $R(j)$ for every $j\in B$ and the columns in $B$ is a minor isomorphic to $C_s^p$. We are interested in identifying, for every $j\in B$, a particular index $r(j) \in R(j)$.

\begin{definition}\label{filas}
For every $j\in [s]$, let $h_j=\min \{u_i-b_j: i\in R(j)\}$ and let $r(j)$ be the element of $R(j)$ for which $u_{r(j)}=b_j+h_j$. 
\end{definition}

We have the following property:

\begin{lemma}\label{ibullets}
If there is $i\in [m]$ such that $\ell_i=b_{j-p}+1$ for some $j\in[s]$ then $r(j)=i$. 
\end{lemma}

\begin{proof}
Let $i\in [m]$ such that $\ell_i=b_{j-p}+1$ for some $j\in[s]$. Since $|B\cap [\ell_i,u_i]|\geq p$, $b_j\in [\ell_i,u_i]_n$ and for every $k\in R(j)$ with $k\neq i$, $\ell_k\in [\ell_i+1, b_{j-p+1}]_n$. 
Since $A$ has no dominating rows, it follows that $u_i\in [b_j, u_k-1]_n$ for all $k\in R(j)$. But then, $i\in R(j)$ and $i=r(j)$. 
\end{proof}

Observe that $\ell_{r(j+p)} \in [b_j+1, b_{j+1}]_n$.  
Since $u_{r(j)} \in  [b_j, b_{j+1}-1]_n$, then $\ell_{r(j+p)} \in [b_j+1, u_{r(j)}]_n$ or $\ell_{r(j+p)} \in [u_{r(j)}+1, b_{j+1}]_n$. Then, we define the following: 

\begin{definition}\label{columnas}
For every $j\in [s]$, let $b_j':= \ell_{r(j+p)}-1$, if $\ell_{r(j+p)} \in [b_j+1,u_{r(j)}]_n$, and $b_j':=u_{r(j)}$, otherwise. Moreover, let $B':=\{b_j':j\in [s]\}$.
\end{definition}

It is clear that if $b_j=u_{r(j)}$ holds for some $j\in[s]$, 
then $b'_j=b_j=u_{r(j)}$. In addition, if $b_j=\ell_{r(j+p)}-1$ holds for some $j\in [s]$, then $b'_j=b_j=\ell_{r(j+p)}-1$. 
The next example shows that we may have $b'_j\neq b_j$ for some $j\in [s]$.    
 
\begin{example} \label{bybprima}
Consider again the circular matrix $A$ defined in \eqref{ejemplo}. Let $B=\{2,5,8,10,12\}$, and $N=[12]\setminus B$. It can be verified that $A/N \approx C_5^2$. Following our notation above, $b_3=8$ and the row $r(3)$ is the third row of $A$, i.e., the row corresponding to $[5,9]_{12}$. Moreover, observe that $b_{3+p}=b_{3+2}=12$ and the row $r(5)$ is the fifth row of $A$,
i.e., the row that corresponds to $[10,12]_{12}$. Thus, we have $\ell_{r(3+2)}- 1=10-1=9$ and $u_{r(3)}=9$. Then it holds that $b'_3=9$. Finally,
it can be verified that $b'_j=b_j$ holds for all $j\in [5]\setminus \{3\}$, i.e., $B'=\{2,5,9,10,12\}$.
\end{example}

%
%
Observe that in the previous example, $b'_1=u_{r(1)}\neq \ell_{r(1+p)}-1= 4$. However,
in the particular case when $A$ is a circulant matrix, we have $b_j= b'_j= \ell_{r(j+p)}-1$ for all $j\in [s]$, as shown in the next remark.

%

\begin{remark}\label{ibullets2}
If $A=C_n^k$, then for every $j\in [s]$ there is a row $i$ such that $\ell_i=b_{j-p}+1$. But then, from 
Lemma~\ref{ibullets}, it follows that $r(j)=i$. Then, 
$\ell_{r(j+p)}=b_j+1$, and thus $b'_j=b_j$.
\end{remark}

In Example~\ref{bybprima}, if we let $N':=[12]\setminus B'$, then we have $A/N \approx A/N' \approx C_5^2$.
We show in the following that it is always the case. 

\begin{lemma}\label{Bprima2}
Let $A$ be a circular matrix, $N\subset [n]$ such that $A/N \approx C_s^p$, and $B=[n]\setminus N$. If $B'$ is constructed as in Definition~\ref{columnas} and $N'=[n]\setminus B'$, then $A/N'\approx C_s^p$.
\end{lemma}



\begin{proof}
It is enough to prove that, for all $i\in [m]$, $[\ell_{i},u_{i}]_n\cap B'=\{b'_j:b_j \in [\ell_{i},u_{i}]_n\}$.

Observe that, from Definition~\ref{columnas}, it follows that  $b_j'\in [b_j,u_{r(j)}]_n$ and $b_j'\notin [\ell_{r(j+p)},u_{r(j+p)}]_n$ hold for each $j\in [s]$. 

Let $j \in [s]$. Assume there exists $i\in [m]$ such that $b_j'\in [\ell_{i},u_{i}]_n$, but $b_j\notin [\ell_{i},u_{i}]_n$.
Then we must have $\{b_{j+1},\ldots,b_{j+p}\}\subset [\ell_{i},u_{i}]_n$, but in this case $\ell_i\in(b_j,b'_j]_n$ and from the definition of $r(j+p)$ we have that $b_j' \in [\ell_{r(j+p)},u_{r(j+p)}]_n$, a contradiction.

Conversely, assume there exists $i\in [m]$ such that $b_j\in [\ell_{i},u_{i}]_n$, but $b'_j\notin [\ell_{i},u_{i}]_n$. Since  $b_j'\in [b_j,u_{r(j)}]_n$ it holds that $b_j\notin [u_{i}+1,u_{r(j)}]_n$ contradicting the definition of $r(j)$. 

\end{proof}

%
%
%

Finally, we present the main contribution of this paper.

\begin{theorem} 
Let $A$ be a circular matrix. Then, $A$ has $C_s^p$ as a circulant minor if and only if there exists a family $\Gamma=\{\Gamma^i: i\in [a]\}$ of $a=\gcd(s,p)$ disjoint circuits in $F(A)$ without bad arcs, each of them having $\frac{s}{a}$ row arcs and winding number $\frac{p}{a}$. Moreover, if $B$ is the set of essential bullets of $\Gamma$ and $N=[n]\setminus B$ then $A/N \approx C_{s}^{p}$.
\end{theorem}

\begin{proof}
From Theorem \ref{th:circulant-minor}, if $\Gamma$ is a collection of $a$ disjoint circuits in $F(A)$, without bad arcs, all of them having $\frac{s}{a}$ row arcs and winding number $\frac{p}{a}$, $B$ is the set of essential bullets of $\Gamma$, and $N=[n]\setminus B$, then $A/N=C_{s}^{p}$. 

Now assume that $A$ has a circulant minor $C_{s}^{p}$, i.e., there is a set $N\subset [n]$ such that $A/N=C_{s}^{p}$. Let $B=[n]\setminus N=\{b_j: j\in [s]\}$, with $1\leq b_1<\ldots<b_s\leq n$. By Lemma \ref{Bprima2}  we may assume that 
$b_j=\ell_{r(j+p)}-1$ or $b_j=u_{r(j)}$ for all $j\in [s]$. 

Consider the set $T$ of row arcs in $F(A)$ defined by $T=\{(\ell_{r(j)}-1, u_{r(j)}): j\in [s]\}$.

Let $P=\{j\in [s]: b_j\neq \ell_{r(j+p)}-1\}$. If $j\in P$, from Definition \ref{columnas}, $b_j=u_{r(j)}$ and there is a path of short forward arcs in $F(A)$ that joins $b_j$ with $\ell_{r(j+p)}-1$. Denote by $F_j$ the set of short arcs of such a path. In addition, since $b_j=u_{r(j)}$ and $b_{j+1}\in[\ell_{r(j+p)}, u_{r(j+p)}]_n$, there is no arc in $T$ that begins or ends in a vertex of $[b_j+1, \ell_{r(j+p)}-2]_n$.

Similarly, define $Q=\{j\in [s]: b_j\neq u_{r(j)}\}$. For every $j\in Q$, it holds that $b_j= \ell_{r(j+p)}-1$ and there is a path of reverse short arcs in $F(A)$ that goes from $u_{r(j)}$ to $\ell_{r(j+p)}-1=b_j$. Let $R_j$ be the set of short arcs of such a path. It is clear that no arc in $T$ begins or ends in a vertex of $[\ell_{r(j+p)}, u_{r(j)}-1]_n$. 

Finally, consider the subgraph $\Gamma$ induced by $T\cup (\cup_{j\in P} F_j)\cup (\cup_{j\in Q} R_j)$ in $F(A)$.

By construction, every vertex in $\Gamma$ has in-degree and out-degree equal to one. Then, $\Gamma$ is a family of $a'$ disjoint circuits in $F(A)$, each one of them having $s'$ row arcs and winding number $p'$, with $\gcd(s',p')=1$. 
Moreover, the set of essential bullets of $\Gamma$ coincides with $B$. Since $A/N \approx C_s^p$, each row arc in $F(A)$ jumps over at least $p$ essential bullets, and $\Gamma$ has no bad arcs. Furthermore, $\Gamma$ has $\left|B\right|=s$ row arcs and each row arc in $\Gamma$ jumps over $p$ essential bullets. Thus, $s=a's'$, $p=a'p'$, and $a'=\gcd(s,p)$.
\end{proof}
Consider the matrix $A$ given in (\ref{ejemplo}) and let $B=\{2,5,9,10,12\}$. We have 
\begin{center}
$T=\{(11,2),(12,5),(4,9),(6,10), (9,12)\}$, $P=\{2,5,10\}$, and $Q=\emptyset$. 
\end{center}
Thus, we have to add only forward short arcs. They correspond to $F_1=\{(2,3), (3,4)\}$, $F_2=\{(5,6)\}$, and $F_4=\{(10,11)\}$. It can be checked that $T\cup F_1 \cup F_2 \cup F_4$ induces a circuit in $F(A)$ whose essential bullets are the vertices in $B$.  

\medskip

As a corollary of the previous theorem and Remark \ref{ibullets2} we have an alternative characterization of circulant minors of circulant matrices to the one given in \cite{Aguilera09}. 

\begin{corollary}
Let $A=C_n^k$ and $D(n,k)$ be the digraph with vertex set $[n]$ and arcs of the form $(i, i+k)$ and $(i,i-1)$ for all $i\in [n]$. Then, $A$ has a circulant minor $C_{s}^{p}$ with $\gcd(s,p)=a$ if and only if there is a family $\Gamma=\{\Gamma^i: i\in [a]\}$ of disjoint circuits in $D(n,k)$, each of them having $\frac{s}{a}$ row arcs and winding number $\frac{p}{a}$. Moreover, if $B$ is the set of essential bullets of $\Gamma$ and $N=[n]\setminus B$, then $C_n^k/N \approx C_{s}^{p}$.
\end{corollary}

\begin{proof}
Let $\Gamma=\{\Gamma^i: i\in [a]\}$ be a family of disjoint circuits in $D(n,k)$. Then, $\Gamma$ is a family of disjoint circuits in $F(C_n^k)$ such that $\circ(\Gamma)=\emptyset$. Thus, $\Gamma$ has no bad arcs and from the previous theorem,  $C_n^k/N \approx C_s^p$. 

Conversely, let $N\subset [n]$ be such that $C_n^k/N \approx C_s^p$. By Remark~\ref{ibullets2} we have $P=\emptyset$ in the proof of the previous theorem. Thus, the family $\Gamma$ of disjoint circuits has no forward short arcs and it is also a family of disjoint circuits in $D(n,k)$. 
\end{proof}

Recall that, given $1\leq k\leq n-1$, the digraph $G(C^k_n)$ defined in \cite{Cornuejols94} has $[n]$ as set of vertices and, for each $i\in [n]$, two arcs leaving $i$: one arc $(i,i+k)$ having length $k$, and one arc $(i,i+k+1)$ having length $k+1$. 
From the results in \cite{Aguilera09} we know that there exist $d$ disjoint circuits in $G(C^k_n)$ if  and only if there exist positive integers $n_1, n_2, n_3$ such that $\gcd(n_1, n_2,n_3)=1$, $n_1n= n_2 k+ n_3 (k+1)$, $d(n_2+n_3)\leq n-2$ and $dn_1\leq k-1$. Similarly, it can be proven that there exist $a$ disjoint circuits in $D(n,k)$ if and only if there exist positive integers $p, s$ and $w$, with $\gcd(s,p)=1$ such that $p n=s k- w$, $a(s+w)\leq n-2$, and $ap\leq k-1$. Then, as a consequence of the previous corollary and the results in \cite{Aguilera09}, 
we prove the following relationship between families of circuits in $D(n,k)$ and in $G(C^k_n)$. 

\begin{theorem}
Let $n,k$ such that $1\leq k\leq n-1$. 
\begin{enumerate}
	\item Let $\mathcal C=\{C^i: i\in [d]\}$ be  a family of disjoint circuits in $G(C_n^k)$, each one with $n_2$ arcs of length $k$, $n_3$ arcs of length $k+1$, and winding number $n_1$, such that $\gcd(n_1, n_2,n_3)=1$. Then there exists a family $\Gamma=\{\Gamma^i: i\in [a]\}$ of disjoint circuits in $D(n,k)$, with $a=\gcd(k-d n_1, n- d(n_2+n_3))$, each one with $s$ row arcs and winding number $p$, where $s=\frac{n-d(n_2+n_3)}{a}$ and $p=\frac{k-d n_1}{a}$.  
	\item Let $\Gamma=\{\Gamma^i: i\in [a]\}$ be a family of disjoint circuits in $D(n,k)$, each one with $s$ row arcs and winding number $p$,  with $\gcd(s,p)=1$. Then, there exists a family $\mathcal C=\{C^i: i\in [d]\}$ of disjoint circuits in $G(C_n^k)$, with $d=\gcd(k-ap,n(ap+1)-as(k+1), a(sk-np))$, each one with $n_2$  arcs of length $k$,  $n_3$ arcs of length $k+1$, and winding number $n_1$, where $n_2=\frac{n(ap+1)-as(k+1)}{d}$, $n_3=\frac{a(sk-np)}{d}$, and $n_1=\frac{k-ap}{d}$.
\end{enumerate}
Moreover, in both cases, 
the set of essential bullets of $\Gamma$ coincides with the set $[n]\setminus V(\mathcal C)$.
\end{theorem}

\begin{proof}
The existence of $d$ disjoint circuits in $G(C_n^k)$, with $n_1$ arcs of length $k$, $n_2$ arcs of length $k+1$, and winding number $n_1$ with $\gcd(n_1, n_2,n_3)=1$ implies that 
$n_1 n= n_2 k + n_3 (k+1)$. 
It can be verified that 
$(k-d n_1)n= (n- d(n_2+n_3))k-dn_3$,
proving the first statement.

The existence of $a$ disjoint circuits in $D(n,k)$, with $s$ row arcs, winding number $p$, and with $\gcd(s,p)=1$ implies that $pn=sk-w$ for some $w\geq 0$. 
It can be verified that
$$(k-ap)n = (n(ap+1)-as(k+1)) k + a(sk-np) (k+1),$$
and the second statement follows.


The relationship between the essential bullets of $\Gamma$ and the vertices of $\mathcal C$ follows from the relationship between these families of circuits in $G(C_n^k)$ and the corresponding circulant minors, proved in \cite{Aguilera09}, and from the relationship between these minors and  the families of circuits in $D(n,k)$, proved in the previous corollary.

\end{proof}

%
%
%
%
%

\end{document}